\documentclass[11pt, notitlepage]{article}   

\usepackage{amsmath,amsthm,amsfonts}   

\usepackage{amscd}
\usepackage{amsfonts}
\usepackage{amssymb}
\usepackage{color}      
\usepackage{epsfig}
\usepackage{graphicx}           

\theoremstyle{plain}
\newtheorem{Thm}{Theorem}[section]
\newtheorem{Lem}[Thm]{Lemma}

\theoremstyle{definition}

\def\finf{\mathop{{\rm I}\kern -.27 em {\rm F}}\nolimits}

\begin{document}

\title{On Domination Number and Distance in Graphs}

\author{{\bf Cong X. Kang}\\
\small Texas A\&M University at Galveston, Galveston, TX 77553, USA\\
{\small\em kangc@tamug.edu}}

\maketitle

\date{}

\begin{abstract}
A vertex set $S$ of a graph $G$ is a \emph{dominating set} if each vertex of $G$ either belongs to $S$ or is adjacent to a vertex in $S$. The \emph{domination number} $\gamma(G)$ of $G$ is the minimum cardinality of $S$ as $S$ varies over all dominating sets of $G$. It is known that $\gamma(G) \ge \frac{1}{3}(diam(G)+1)$, where $diam(G)$ denotes the diameter of $G$. Define $C_r$ as the largest constant such that $\gamma(G) \ge C_r \sum_{1 \le i < j \le r}d(x_i, x_j)$ for any $r$ vertices of an arbitrary connected graph $G$; then $C_2=\frac{1}{3}$ in this view. The main result of this paper is that $C_r=\frac{1}{r(r-1)}$ for $r\geq 3$. It immediately follows that $\gamma(G)\geq \mu(G)=\frac{1}{n(n-1)}W(G)$, where $\mu(G)$ and $W(G)$ are respectively the average distance and the Wiener index of $G$ of order $n$. As an application of our main result, we prove a conjecture of DeLaVi\~{n}a et al.\;that $\gamma(G)\geq \frac{1}{2}(ecc_G(B)+1)$, where $ecc_G(B)$ denotes the eccentricity of the boundary of an arbitrary connected graph $G$.
\end{abstract}


\section{Introduction}

We consider finite, simple, undirected, and connected graphs $G = (V(G),E(G))$ of order $|V(G)|\ge 2$ and size $|E(G)|$. For $W \subseteq V(G)$, we denote by $\langle W \rangle_G$ the subgraph of $G$ induced by $W$. For $v \in V(G)$, the \emph{open neighborhood} of $v$ is the set $N_G(v)=\{u \in V(G) \mid uv \in E(G)\}$, and the \emph{closed neighborhood} of $v$ is $N_G[v]=N_G(v) \cup \{v\}$. Further, let $N(S)=\cup_{v\in S}N(v)$ and $N[S]=\cup_{v\in S}N[v]$ for $S\subseteq V(G)$. The degree of a vertex $v \in V(G)$ is $\deg_G(v)=|N_G(v)|$.  The \emph{distance} between two vertices $x, y \in V(G)$ in the subgraph $H$, denoted by $d_H(x, y)$, is the length of the shortest path between $x$ and $y$ in the subgraph $H$. The \emph{diameter} $diam(H)$ of a graph $H$ is $\max\{d_H(x,y) \mid x,y \in V(H)\}$.\\

A set $S \subseteq V(G)$ is a \emph{dominating set} (resp. \emph{total dominating set}) of $G$ if $N[S]=V(G)$ (resp. $N(S)=V(G)$). The \emph{domination number} (resp. \emph{total domination number}) of $G$, denoted by $\gamma(G)$ (resp. $\gamma_t(G)$), is the minimum cardinality of $S$ as S varies over all dominating sets (resp. total dominating sets) in $G$; a dominating set (resp. total dominating set) of $G$ of minimum cardinality is called a $\gamma(G)$-set (resp. $\gamma_t(G)$-set). \\

Both distance and (total) domination are very well-studied concepts in graph theory. For a survey of the myriad variations on the notion of domination in graphs, see~\cite{Dom1}.\\

It is well-known that $\gamma(G) \ge \frac{1}{3}(diam(G)+1)$ ($*$); a ``proof\," to ($*$) can be found on p.56 of the authoritative reference~\cite{Dom1}. However, the ``proof\," contained therein is logically flawed. We provide a counter-example to a crucial assertion in the ``proof\," and then present a correct proof to ($*$). Upon some reflection, we see that ($*$) is the two parameter case of a family of inequalities existing between $\gamma(G)$ and the distances in $G$, in the following way: $\gamma(G)\geq \frac{1}{3}(diam(G)+1)=\frac{1}{3{r \choose 2}}\left({r \choose 2}diam(G)+{r \choose 2}\right)\geq \frac{1}{3{r \choose 2}}\left(\sum_{1\leq i<j\leq r}d(x_i,x_j)\right)$. The inequality $\gamma(G)\geq \frac{1}{3{r \choose 2}}\left(\sum_{1\leq i<j\leq r}d(x_i,x_j)\right)$ naturally brings up the question: What is the largest constant $C_r$, such that $\gamma(G)\geq C_r\left(\sum_{1\leq i<j\leq r}d(x_i,x_j)\right)$, for all connected graphs $G=(V,E)$ and arbitrary vertices $x_1, \ldots, x_r \in V$, where $r\geq 2$? Taking this viewpoint, we have $C_2=\frac{1}{3}$ by $(*)$.\\

The main result of this paper is that $C_r=\frac{1}{r(r-1)}$ for $r\geq 3$. Since, for a graph $G$ of order $n$, $W(G)=\sum_{1\leq i<j\leq n}d(x_i,x_j)$ is the \emph{Wiener index} of $G$ (see~\cite{Wiener}) and $\mu(G)=\frac{1}{n(n-1)}W(G)$ is the average distance (per definition found in~\cite{Dank}), 
it follows that $\gamma(G)\geq \mu(G)=\frac{1}{n(n-1)}W(G)$. As an application of our main result, we prove a conjecture in~\cite{Vina-lower} by DeLaVi\~{n}a et al.\;that $\gamma(G)\geq \frac{1}{2}(ecc_G(B)+1)$, where $ecc_G(B)$ denotes the eccentricity of the boundary of an arbitrary connected graph $G$ (to be defined in Section 4). \\

This paper is motivated by the work of Henning and Yeo in~\cite{HY}, where they obtained similar inequalities for total domination number $\gamma_t$ (rather than domination number $\gamma$). Given the close relation between the two graph parameters, we expect the techniques used in~\cite{HY} to be readily adaptable towards the results of this paper. However, in striking contrast to~\cite{HY}, we avoid the painstaking case-by-case, structural analysis employed there by making use of the easy and well-known Lemma~\ref{stree}; this results in a much simpler and shorter paper. Further, we are able to obtain (in domination) the exact value of $C_r$ for every $r$, rather than only a bound (in total domination, c.f.~\cite{HY}) for $C_r$ for all but the first few values of $r$.



\section{An Error in the proof of $\gamma(G) \ge \frac{1}{3}(diam(G)+1)$ in FoDiG}

For readers' convenience, we first reproduce Theorem 2.24 and its incorrect proof as it appears on p.56 of~\cite{Dom1}, the authoritative reference in the field of domination titled \textit{Fundamentals of Domination in Graphs}.

\begin{Thm}\label{gamma-diam}
For any connected graph $G$, $\left\lceil\frac{diam(G)+1}{3}\right\rceil\leq \gamma(G)$.
\end{Thm}

``Proof\," (as found on p.56 of~\cite{Dom1}).
Let $S$ be a $\gamma$-set of a connected graph $G$. Consider an arbitrary path of length $diam(G)$. This diametral path includes at most two edges from the induced subgraph $\langle N[v]\rangle$ for each $v\in S$. Furthermore, since $S$ is a $\gamma$-set, the diametral path includes at most $\gamma(G)-1$ edges joining the neighborhoods of the vertices of $S$. Hence, $diam(G)\leq 2\gamma(G)+\gamma(G)-1=3\gamma(G)-1$ and the desired result follows."~\hfill $\Box$ \\

\begin{figure}[htpb]
\begin{center}
\includegraphics[width=.2\textwidth]{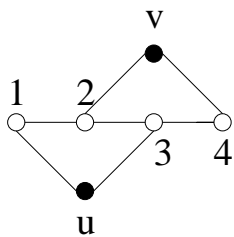}
\caption{a counter-example}\label{figure1}
\end{center}
\end{figure}

Presumably, by a ``diametral path", the authors had in mind an induced path with length $diam(G)$. Still, the assertion of the sentence beginning with ``Furthermore" is incorrect, as seen by the example in Figure~\ref{figure1}: notice that $S=\{u,v\}$ is a $\gamma$-set and the vertices $1,2,3,4$ form a diametral path containing $3$ edges joining $\langle N[u]\rangle$ with $\langle N[v]\rangle$, whereas $\gamma(G)-1=1$.


\section{Domination number and distance in graphs}

The following lemma can be proved by exactly the same argument given in the proof of Lemma 2 in~\cite{DeLaVina}; it was also observed on p.23 of~\cite{Dank}.

\begin{Lem}\cite{Dank, DeLaVina}\label{stree}
Let $M$ be a $\gamma(G)$-set. Then there is a spanning tree $T$ of $G$ such that $M$ is a $\gamma(T)$-set.
\end{Lem}

Now, we apply Lemma~\ref{stree} to give a correct proof of Theorem~\ref{gamma-diam}. \\

Proof of Theorem~\ref{gamma-diam}.
Given $G$, take a spanning tree $T$ of $G$ such that $\gamma(G)=\gamma(T)$. Suppose, for the sake of contradiction, $\gamma(G)<\frac{1}{3}(diam(G)+1)$. Since $\gamma(T)=\gamma(G)$ and $diam(T) \ge diam(G)$, we have
\begin{equation}\label{eq1}
\gamma(T) < \frac{1}{3}(diam(T)+1)
\end{equation}
Take a path $P$ of $T$ with length equal to $diam(T)$. If (\ref{eq1}) holds, there must exist a vertex $u$ of $T$ such that $|V(P) \cap N[u]| \ge 4$. Since $P$ is a path of $T$ (a tree), this is impossible.~\hfill $\Box$\\

\begin{Thm}\label{c3}
Given any three vertices $x_1,x_2,x_3$ of a connected graph $G$, we have
\begin{equation}\label{eq2}
\gamma(G) \ge \frac{1}{6}(d_G(x_1, x_2)+d_G(x_1, x_3)+d_G(x_2, x_3)).
\end{equation}
Further, if equality is attained in~(\ref{eq2}), then $d_G(u, v)\equiv 2 \pmod 3$ for any pair $u,v\in\{x_1,x_2,x_3\}$.
\end{Thm}

\begin{proof}
By Lemma~\ref{stree}, there exists a spanning tree $T$ of $G$ with $\gamma(T)=\gamma(G)$. Since $d_T(u,v) \ge d_G(u,v)$ for any two vertices $u,v \in V(T)=V(G)$, it suffices to prove (\ref{eq2}) on $T$. If $x_1,x_2$, and $x_3$ all lie on one geodesic, then the inequality~(\ref{eq3}) obviously holds by Theorem~\ref{gamma-diam}. Thus, let $d_T(x_1, y)=a$, $d_T(x_2, y)=b$, and $d_T(x_3, y)=c$, with $0\notin \{a,b,c\}$, as shown in Figure~\ref{figure2}. Then, the inequality~(\ref{eq2}) on $T$ becomes
\begin{equation}\label{eq3}
\gamma(T)\geq\frac{1}{6}((a+b)+(a+c)+(b+c))=\frac{1}{3}(a+b+c).
\end{equation}
 Let $y'$ be the vertex lying on the $x_2$-$y$ path and adjacent to $y$. Let $P^1$ and $P^2$ denote the $x_1$-$x_3$ path and the $x_2$-$y'$ path, respectively. If there exists a $\gamma(T)$-set $M$ Not containing $y$, then $M$ must contain a neighbor $z$ of $y$. Suppose, WLOG, $z\neq y'$. Then, inequality~(\ref{eq3}) follows immediately from applying Theorem~\ref{gamma-diam} to $P^1$ and $P^2$. If $y$ belongs to every $\gamma(T)$-set $M$, then $\gamma(T) \ge 1+\frac{1}{3}(a-1)+\frac{1}{3}(b-1)+\frac{1}{3}(c-1)=\frac{1}{3} (a+b+c)$, and (\ref{eq3}) again follows.

\begin{figure}[htpb]
\begin{center}
\includegraphics[width=.4\textwidth]{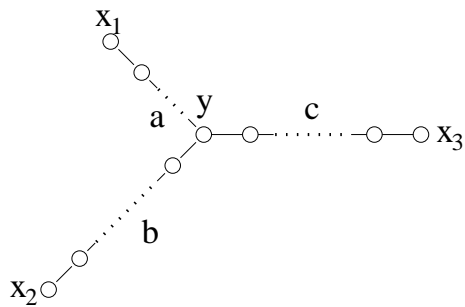}
\caption{$r=3$ case}\label{figure2}
\end{center}
\end{figure}

Next, suppose equality is attained in~(\ref{eq2}). Again, let $T$ be a spanning tree with $\gamma(T)=\gamma(G)$. Since $d_G(x_i,x_j)\leq d_T(x_i,x_j)$ and $\gamma(T)\geq\frac{1}{6}(d_T(x_1, x_2)+d_T(x_1, x_3)+d_T(x_2, x_3))$ holds, we have $\gamma(G)=\frac{1}{6}(d_G(x_1, x_2)+d_G(x_1, x_3)+d_G(x_2, x_3))\leq\frac{1}{6}(d_T(x_1, x_2)+d_T(x_1, x_3)+d_T(x_2, x_3))\leq \gamma(T)$. Thus, we deduce that $\gamma(T)=\frac{1}{6}(d_T(x_1, x_2)+d_T(x_1, x_3)+d_T(x_2, x_3))$ and $d_G(x_i,x_j)=d_T(x_i,x_j)$ for each pair $(x_i,x_j)$. With $a, b, c$ defined as above, the present assumption implies $\gamma(T)=\frac{1}{3}(a+b+c)$. Observe, in light of Theorem~\ref{gamma-diam}, that the equality $\gamma(T)=\frac{1}{3}(a+b+c)$ is only possible if the following ``optimal domination" of $T$ occurs: there is a $\gamma(T)$-set $M$ which contains $y$, a degree-three vertex in $\langle V(P^1)\cup V(P^2) \rangle_T$ which dominates four or more vertices in $T$; every other vertex of $M$ dominates three or more vertices in $T$; no vertex of $T$ is dominated by more than one vertex of $M$. (Note that Figure~\ref{figure2} only shows $\langle V(P^1)\cup V(P^2) \rangle_T$, which may be a strict subgraph of $T$.) This ``optimal domination" condition clearly implies that each member of $\{a,b,c\}$ must equal $1\pmod 3$, which yields our second claim. ~\hfill
\end{proof}

Next, we determine the largest $C_r$ for $r \ge 3$ with the method deployed in~\cite{HY}. However, rather than just getting a bound on $C_r$ in the case of total domination there, we obtain the exact value of $C_r$ for every $r$.

\begin{Thm}
For $r \ge 3$, $C_r=\displaystyle\frac{1}{r^2-r}$.
\end{Thm}

\begin{proof}
First, we prove $C_r \le \displaystyle\frac{1}{r^2-r}$. Let $G=K_{1,r}$ be a star with $r$ leaves labeled $x_1,\ldots, x_r$. Then $\gamma(G)=1$ and $$\sum_{1\le i < j \le r}d(x_i, x_j)={r  \choose 2} \cdot 2=r(r-1).$$ So, $C_r \le \frac{1}{r(r-1)}$.

Next, we show that $C_r \ge \displaystyle\frac{1}{r^2-r}$. Notice that $C_3=\frac{1}{3(3-1)}=\frac{1}{6}$ is given by Theorem~\ref{c3}. Thus, let $x_1, x_2, \ldots, x_r$ be any arbitrary $r\geq 4$ vertices of $G$. Since $\gamma(G) \ge \frac{1}{6}(d_G(x_i, x_j)+d_G(x_i, x_k)+d_G(x_j, x_k))$ holds for any triplet $\{x_i, x_j, x_k\}\subseteq \{x_1, x_2, \dots, x_r\}$, we have
$${r \choose 3} \gamma(G) \ge \sum_{1 \le i < j < k \le r}\frac{1}{6}(d_G(x_i, x_j)+d_G(x_i, x_k)+d_G(x_j, x_k))=\frac{r-2}{6}\sum_{1 \le i < j \le r}d(x_i, x_j)\, ;$$
note that the last equality comes from the fact that there are $r\!-\!2$ triplets containing any given pair of vertices. Thus, $\displaystyle C_r \ge\frac{1}{{r \choose 3}}\frac{r-2}{6}=\frac{1}{r(r-1)}$ as well.~\hfill
\end{proof}

\section{Applying Theorem~\ref{c3} to a Conjecture of DeLaVi\~{n}a et al.}

We need a few more definitions. The \emph{eccentricity} of a vertex $v$ in $G$, denoted by $ecc_G(v)$, is $\max\{d_G(v,x)\;|\;x\in V(G)\}$. The \emph{boundary} of $G$ is defined as the set $B(G)=\{v\in V(G)\;|\;ecc_G(v)=diam(G)\}$; we denote it simply as $B$ hereafter. The distance between a vertex $v\in V(G)$ and a set $S\subseteq V(G)$ is defined as $d_G(v,S)=\min\{d_G(v,x)\;|\;x\in S\}$. Further, the eccentricity of $S\subseteq V(G)$ is defined as $ecc_G(S)=\max\{d_G(x,S)\;|\;x\in V(G)\}$. \\

In~\cite{Vina-lower}, DeLaVi\~{n}a et al.\;proved, \emph{for a tree $G$}, that $\gamma(G)\geq \frac{1}{2}(ecc_G(B)+1)$. They further conjectured that the inequality holds for any connected graph $G$. As an application of Theorem~\ref{c3}, we prove this conjecture. Our proof follows the arguments given by Henning and Yeo in~\cite{HY} proving the analogous Graffiti.pc conjecture $\gamma_t(G)\geq \frac{2}{3}(ecc_G(B)+1)$.\\

\begin{Thm}\label{gamma-eccen}
Let $G$ be a connected graph. Then $\displaystyle\gamma(G)\geq \frac{1}{2}(ecc_G(B)+1)$.
\end{Thm}

\begin{proof}
If $B=V(G)$, then $ecc_G(B)=0$ and the desired inequality obviously holds. So, suppose $B\neq V(G)$; this implies that $|V(G)|\geq 3$ and $|B|\geq 2$. Pick vertices $x$ and $y$ with $d(x,y)=diam(G)$; then, $x,y\in B$. Let $ecc_G(B)=R$. Pick $z\in V(G)\setminus B$ such that $d(z,B)=R$. We have $d(x,z)\geq R, \ d(y,z)\geq R$ and $d(x,y)=diam(G)\geq R+1$. Hence, we have $d(x,y)+d(x,z)+d(y,z)\geq 3R+1$ $(\spadesuit)$. If equality holds in $(\spadesuit)$, then $R=d(x,z)=d(y,z)=d(x,y)-1$, and we can Not have both $d(x,z)$ and $d(x,y)$ be congruent to $2$ mod $3$. In this case, by Theorem~\ref{c3}, we have that $\gamma(G)>\frac{1}{6}(d(x,y)+d(x,z)+d(y,z))=\frac{1}{6}(3R+1)=\frac{1}{2}R+\frac{1}{6}$, which implies $\gamma(G)\geq\frac{1}{2}R+\frac{1}{2}$. On the other hand, if the inequality $(\spadesuit)$ is strict, again by Theorem~\ref{c3}, we have that $\gamma(G)\geq\frac{1}{6}(d(x,y)+d(x,z)+d(y,z))>\frac{1}{6}(3R+1)=\frac{1}{2}R+\frac{1}{6}$, which again implies $\gamma(G)\geq\frac{1}{2}R+\frac{1}{2}$.~\hfill
\end{proof}


\end{document}